\documentclass[10pt]{amsart}
\usepackage{amssymb}
\usepackage{amsthm,amsmath}
\usepackage{cite}
\usepackage{graphicx}
\usepackage{MnSymbol}

\title[Repeated Principal Indefinite Summation]{Repeated Principal Indefinite Summation}

\author{Thomas Lamby}
\address{University of Luxembourg, Department of Mathematics, Maison du Nombre, 6, avenue de la Fonte, L-4364 Esch-sur-Alzette, Luxembourg}\thanks{Corresponding author: Thomas Lamby, University of Luxembourg, Department of Mathematics, Maison du Nombre, 6, avenue de la Fonte, L-4364 Esch-sur-Alzette, Luxembourg. Email: thomas.lamby[at]uni.lu}
\email{thomas.lamby[at]uni.lu}

\author{Jean-Luc Marichal}
\address{University of Luxembourg, Department of Mathematics, Maison du Nombre, 6, avenue de la Fonte, L-4364 Esch-sur-Alzette, Luxembourg}
\email{jean-luc.marichal[at]uni.lu}

\date{\today}

\begin{document}

\theoremstyle{plain}

\newtheorem{theorem}{Theorem}[section]
\newtheorem{lemma}[theorem]{Lemma}
\newtheorem{proposition}[theorem]{Proposition}
\newtheorem{corollary}[theorem]{Corollary}
\newtheorem{fact}[theorem]{Fact}

\theoremstyle{definition}

\newtheorem{definition}[theorem]{Definition}

\newtheorem{ex}[theorem]{Example}
    \newenvironment{example}
    {\renewcommand{\qedsymbol}{$\lozenge$}
    \pushQED{\qed}\begin{ex}}
    {\popQED\end{ex}}

\newtheorem{rem}[theorem]{Remark}
    \newenvironment{remark}
    {\renewcommand{\qedsymbol}{$\lozenge$}
    \pushQED{\qed}\begin{rem}}
    {\popQED\end{rem}}

\theoremstyle{remark}

\newtheorem*{claim}{Claim}

\newcommand{\R}{\mathbb{R}}
\newcommand{\N}{\mathbb{N}}
\newcommand{\Z}{\mathbb{Z}}
\newcommand{\C}{\mathbb{C}}
\newcommand{\id}{\mathrm{id}}
\newcommand{\cC}{\mathcal{C}}
\newcommand{\cD}{\mathcal{D}}
\newcommand{\cK}{\mathcal{K}}
\newcommand{\dom}{\mathrm{dom}\,}
\newcommand{\ran}{\mathrm{ran}\,}

\def\tchoose#1#2{{\textstyle{{{#1}\choose{#2}}}}}

\newcommand{\jlm}[1]{\marginpar{\hrule\medskip JL: #1}}

\begin{abstract}
Under suitable asymptotic and convexity conditions on a function $g\colon\R_+\to\R$, the solution to $\Delta f=g$, where $\Delta$ is the forward difference operator, is unique up to an additive constant and is called the principal indefinite sum of $g$, generalizing the additive form of Bohr-Mollerup's theorem. We consider the map $\Sigma$, which assigns to each admissible function $g$ its principal indefinite sum that vanishes at $1$, and we naturally explore its iterates, which produce repeated principal indefinite sums, in analogy with the concept of repeated indefinite integrals. Explicit formulas and convergence results are established, highlighting connections with classical combinatorial and special functions, including the multiple gamma functions, for which we also provide integral representations.
\end{abstract}



\maketitle
\section{Introduction}

Let $g\colon \R_+\to\R$ be a real-valued function defined on the open half-line $\R_+=(0,\infty)$. The class of functions $f\colon \R_+\to\R$ that are solutions to the functional equation $\Delta f=g$ on $\R_+$, where $\Delta$ denotes the classical forward difference operator, is usually referred to as the \emph{indefinite sum} of $g$; see, for example, Graham \emph{et al}.~\cite[p.\ 48]{GraKnuPat94}. This class has the property that any two such solutions differ by a $1$-periodic function.

For instance, if $g$ is the logarithm function, $g(x)=\ln x$, then the indefinite sum of $g$ consists of all the functions of the form
\begin{equation}\label{eq:11tzi}
f(x) ~=~ \ln\Gamma(x)+\omega(x)\qquad (x>0),
\end{equation}
where $\ln\Gamma(x)$ is the log-gamma function and $\omega\colon\R_+\to\R$ is a $1$-periodic function.

Under certain assumptions on the function $g$, its indefinite sum contains a distinguished subclass known as the \emph{principal indefinite sum} of $g$, whose definition we now recall.

Let $\N$ denote the set of nonnegative integers. For any $p\in\N$, we let $\cD^p$ denote the class of functions $f\colon \R_+\to\R$ for which the sequence $n\mapsto \Delta^pf(n)$ converges to zero. We also let $\cK^p$ denote the class of functions $f\colon \R_+\to\R$ that are eventually $p$-convex or eventually $p$-concave (see Definition~\ref{de:pconvpconc}).

Marichal and Zena\"{\i}di \cite[Chapter 5]{MarZen22} established the following result, which may be viewed as a generalization of the additive form of the Bohr-Mollerup theorem. \emph{If the function $g\colon \R_+\to\R$ lies in $\cD^p\cap\cK^p$ for some $p\in\N$, then there is a unique (up to an additive constant) solution $f\colon \R_+\to\R$ in $\cK^p$ to the equation $\Delta f=g$}. Moreover, the special solution that vanishes at $1$ is denoted $\Sigma g$ and can be expressed via the pointwise limit
\begin{equation}\label{eq:lim1}
\Sigma g(x) ~=~ \lim_{n\to\infty}f_n^p[g](x)\qquad (x>0),
\end{equation}
where
\begin{equation}\label{eq:fnp2}
f_n^p[g](x) ~=~ \sum_{k=1}^{n-1}g(k)-\sum_{k=0}^{n-1}g(x+k)+\sum_{j=1}^p\tchoose{x}{j}\,\Delta^{j-1} g(n).
\end{equation}

In this context, the \emph{principal indefinite sum} \cite[Definition 5.4]{MarZen22} of $g$ is the class of functions of the form $\Sigma g(x)+c$, where $c\in\R$. Principal indefinite sums actually constitute a broad family of functions that arise frequently in mathematical analysis. Various examples are presented in detail in \cite[Chapters 10–12]{MarZen22} and \cite{MarZen24}.

Considering again the logarithm function $g(x)=\ln x$, we retrieve the additive form of Bohr-Mollerup's theorem \cite{BohMol22}, which essentially states that the only solutions $f\colon\R_+\to\R$ to the equation $\Delta f=g$ on $\R_+$ that are eventually $1$-convex or eventually $1$-concave (that is, eventually convex or eventually concave) are of the form
$$
f(x) ~=~ \Sigma g(x)+c ~=~ \ln\Gamma(x)+c\qquad (x>0),
$$
where $c\in\R$. These functions therefore constitute the \emph{principal indefinite sum} of the logarithm function, namely the subclass of the indefinite sum \eqref{eq:11tzi} obtained by requiring the $1$-periodic functions $\omega$ to be constant. With a slight abuse of language, we say that the principal indefinite sum of the logarithm function is the log-gamma function. Moreover, as we can easily verify, Eqs.~\eqref{eq:lim1} and \eqref{eq:fnp2} then provide the additive version of Gauss' definition of the gamma function:
\begin{equation}\label{eq:GaussLimit00}
\Gamma(x) ~=~ \lim_{n\to\infty}\,\frac{n!\, n^x}{x(x+1)\,\cdots\, (x+n)}\qquad (x>0).
\end{equation}

Thus defined, the symbol $\Sigma$ denotes the map that assigns to each function $g$ lying in
\begin{equation}\label{eq:domS11}
\dom\Sigma ~=~ \bigcup_{p\geq 0}(\cD^p\cap\cK^p)
\end{equation}
the function $\Sigma g\colon\R_+\to\R$ defined by Eqs.~\eqref{eq:lim1} and \eqref{eq:fnp2}.

In this paper, we are interested in functions $g\colon\R_+\to\R$ for which the map $\Sigma$ can be applied repeatedly to produce ``repeated principal indefinite sums,'' in full analogy with the concept of repeated indefinite integrals.

To this end, we define the iterates of $\Sigma$ in the following classical way, where $\mathrm{id}$ denotes the identity map:
$$
\Sigma^0 ~=~ \mathrm{id},\qquad\Sigma^1 ~=~ \Sigma,\qquad\text{and}\qquad\Sigma^{m+1} ~=~ \Sigma\circ\Sigma^m\qquad (m\in\N).
$$
For instance, it is straightforward to verify that if $g(x)=x$ is the identity function, then, for any $m\in\N$, we have
$$
\Sigma^mg(x) ~=~ \tchoose{x}{m+1}\qquad (x>0),
$$
and it is not difficult to show that this latter function lies in $\cD^{m+2}\cap\cK^{m+2}$ (for details, see Proposition~\ref{prop:pol2} below).

This paper is organized as follows. In Section 2, we review preliminary results and recall the concept of multiple gamma functions, which serve as a motivating example of repeated principal indefinite sums. Sections 3 and 4 extend classical results---the analogs of Gauss' limit and Bohr-Mollerup's theorem---to the setting of repeated principal indefinite sums, thereby enabling us to characterize these functions. This characterization, in turn, allows us in Section 5 to establish unexpected integral representations for the logarithm of multiple gamma functions. Section 6 presents examples of repeated principal indefinite sums, including those arising from certain rational functions. To this end, we derive a general formula for the repeated principal indefinite sums of functions obtained via certain linear operators, including the derivative operator. Finally, in Section 7, we derive an alternative explicit form of repeated principal indefinite sums by first developing a discrete analog of Taylor's theorem.

Throughout this paper, we adopt the following conventional notation. When the derivative operator $D$ is applied to a function $g\colon\R_+\to\R$, its value at a point $x$ is written as $Dg(x)$ rather than $(Dg)(x)$. If $D$ is applied directly to an expression, for example $x^2+x$, the resulting function is naturally denoted $D_x(x^2 + x)$ instead of $D(x^2 + x)$. This convention is used for all maps on functions from $\mathbb{R}_+$ to $\mathbb{R}$ considered in this paper, including the difference operator $\Delta$, the principal indefinite sum map $\Sigma$, and their iterates.

\section{Preliminaries}

In this section, we recall several basic definitions and properties that will be used throughout this paper, and we present a notable example that motivates the subsequent discussion: the \emph{multiple gamma functions}.

Let us first recall the definition of the higher-order convexity properties in terms of divided differences. For background, see for instance \cite[Section 2.2]{MarZen22} and the references therein.

\begin{definition}\label{de:pconvpconc}
Let $I$ be any real interval and let $p\geq -1$ be an integer. A function $f\colon I\to\R$ is said to be \emph{convex of order $p$}, or simply, \emph{$p$-convex}, if the divided difference of $f$ at any $p+2$ pairwise distinct points $x_0,x_1,\ldots,x_{p+1}$ in $I$ is positive; that is,
$$
f[x_0,x_1,\ldots,x_{p+1}] ~\geq ~ 0.
$$
The function $f\colon I\to\R$ is said to be \emph{concave of order $p$}, or simply, \emph{$p$-concave} if its opposite $-f$ is $p$-convex.
\end{definition}

For any real interval $I$ and any integer $p\geq -1$, we let $\cK^p_1(I)$ (resp.\ $\cK^p_{-1}(I)$) denote the class of functions $f\colon I\to\R$ that are $p$-convex (resp.\ $p$-concave). Similarly, we let $\cK_1^p$ (resp.\ $\cK_{-1}^p$) denote the class of functions $f\colon\R_+\to\R$ that are eventually $p$-convex (resp.\ eventually $p$-concave). We also let
$$
\cK^p(I) ~=~ \cK^p_1(I)\cup\cK^p_{-1}(I),\qquad\cK^p ~=~ \cK^p_1\cup\cK^p_{-1}\, .
$$
With these definitions, $\cK^{-1}(I)$ consists of functions on $I$ that keep a constant sign. Likewise, $\cK^0(I)$ is the class of functions on $I$ that are monotone, while $\cK^1(I)$ consists of functions on $I$ that are convex or concave.

Interestingly, it follows readily from Eqs.~\eqref{eq:lim1} and \eqref{eq:fnp2} that if $g$ lies in $\cD^p\cap\cK^p_1$ for some $p\in\N$, then $\Sigma g$ lies in $\cK^p_{-1}$; see also \cite[Chapter 5]{MarZen22}.

Before proceeding further, we also recall the following fundamental propositions \cite[Sections 2.2 and 4.2]{MarZen22}, which will be used frequently throughout this paper. For any open real interval $I$, we let $\cC^p(I)$ denote the space of all real-valued functions defined on $I$ that are $p$ times continuously differentiable.

\begin{proposition}\label{prop:ffprime32}
Let $I$ be an open real interval, let $f\colon I\to\R$ be a function, and let $p\in\N$. Then the following assertions hold:
\begin{itemize}
\item We have $\cK^{p+1}(I) \subset \cC^p(I)$.
\item If $I$ is right-unbounded and $f\in\cK^p_1(I)$, then $\Delta f\in\cK^{p-1}_1(I)$.
\item If $f$ is differentiable, then $f\in\cK^p_1(I)$ if and only if $f'\in\cK^{p-1}_1(I)$.
\end{itemize}
\end{proposition}

\begin{proposition}\label{prop:ffprime324}
Let $f\colon\R_+\to\R$ be a function and let $p\in\N$. Then the following assertions hold:
\begin{itemize}
\item If $f\in\cD^{p+1}\cap\cK^{p+1}_1$, then $\Delta f\in\cD^p\cap\cK^p_1$
\item If $f$ is differentiable, then $f\in\cD^{p+1}\cap\cK^{p+1}_1$ if and only if $f'\in\cD^p\cap\cK^p_1$.
\end{itemize}
\end{proposition}

It is also important to recall \cite[Proposition 4.7]{MarZen22} that the classes $\cD^p$ and $\cK^p$ are nested in opposite directions:
$$
\cD^0\,\subset\,\cD^1\,\subset\,\cD^2\,\subset\,\cdots\, ,\qquad \cK^{-1}\,\supset\,\cK^0 \,\supset\,\cK^1\,\supset\,\cdots\, .
$$
These inclusions naturally lead us to introduce the following limiting sets:
$$
\cD^{\infty} ~=~ \lim_{p\to\infty}\cD^p\, ,\qquad\cK^{\infty} ~=~ \lim_{p\to\infty}\cK^p.
$$

Now, when defining the iterates of the map $\Sigma$, the first question that naturally arises is how to determine the domain of $\Sigma^{m+1}$ for any $m\in\N$. In this regard, we have the following proposition.

\begin{proposition}\label{prop:338-domSm}
Let $g\colon\R_+\to\R$ and let $m\in\N$. The following assertions are equivalent:
\begin{itemize}
\item[(i)] $g\in\dom\Sigma^{m+1}$.
\item[(ii)] $g\in\dom\Sigma^{m}$ and $\Sigma^m g\in\dom\Sigma$.
\item[(iii)] There exists $p\in\N$ such that
\begin{itemize}
    \item[$\bullet$] $g\in\cD^p\cap\cK^p$ \hspace{1ex}(hence $g\in\dom\Sigma$),
    \item[$\bullet$] $\Sigma g\in\cD^{p+1}\cap\cK^{p+1}$ \hspace{1ex}(hence $\Sigma g\in\dom\Sigma$),
    \item[$\bullet$] $\cdots$
    \item[$\bullet$] $\Sigma^m g\in\cD^{p+m}\cap\cK^{p+m}$  \hspace{1ex}(hence $\Sigma^m g\in\dom\Sigma$).
\end{itemize}
\end{itemize}
\end{proposition}

\begin{proof}
It is evident that the assertions (i) and (ii) are equivalent. Moreover, we can readily see that assertion (iii) implies assertion (i). Let us now show that assertion (i) implies assertion (iii). Clearly, we can assume that $m\geq 1$.

Let $g\in\dom\Sigma^{m+1}$. In particular, we have $g\in\dom\Sigma$, and hence $g\in\cD^p\cap\cK^p$ for some $p\in\N$. Let $p$ be the smallest integer with this property; thus $\Sigma g\notin\cD^p$.

On the other hand, since $\Sigma g\in\dom\Sigma$, there exists an integer $q\geq p+1$ such that
$$
\Sigma g ~\in ~ \cD^q\cap\cK^q.
$$
Because $\Sigma g\in\cD^{p+1}$ and $\cK^q\subset\cK^{p+1}$, it follows that
$$
\Sigma g ~\in ~ \cD^{p+1}\cap\cK^{p+1}.
$$
This argument can be easily iterated.
\end{proof}

The equivalence between assertions (i) and (iii) in Proposition~\ref{prop:338-domSm} yields an explicit description of the domain of $\Sigma^{m+1}$ for any $m\in\N$, namely:
$$
\dom\Sigma^{m+1} ~=~ \bigcup_{p\geq 0}~\bigcap_{k=0}^m ~ (\Sigma^k)^{-1}\big(\cD^{p+k}\cap\cK^{p+k}\big).
$$
When $m=0$, we immediately retrieve the description of $\dom\Sigma$ given in \eqref{eq:domS11}.

Although explicit, this characterization of $\dom\Sigma^{m+1}$ relies on preimages under the map $\Sigma^k$ and is therefore not particularly convenient for determining whether a given function $g\colon\R_+\to\R$ lies in this domain. This motivates the search for simple (sufficient) conditions for membership in $\dom\Sigma^{m+1}$. In the following proposition, we show that one such condition can be obtained from the following inclusion:
\begin{equation}\label{eq:SuffCondDomKm}
\bigcup_{p\geq 0}(\cD^p\cap\cK^{p+m})~\subset ~ \dom\Sigma^{m+1}.
\end{equation}

\begin{proposition}\label{prop:SuffCondDomKm44}
Let $m\in\N$. If $g$ lies in $\cD^p\cap\cK^{p+m}$ for some $p\in\N$, then $\Sigma^{m+1}g$ exists and lies in $\cD^{p+m+1}\cap\cK^{p+m}$.
\end{proposition}

\begin{proof}
Suppose that $g$ lies in $\cD^p\cap\cK^{p+m}$ for some $p\in\N$. Then $g\in\cD^{p+k}\cap\cK^{p+k}$ for each $k=0,\ldots,m$. It follows immediately that assertion (iii) of Proposition~\ref{prop:338-domSm} is satisfied, which establishes the claim.
\end{proof}

Throughout this paper, we will frequently make use of the inclusion given in \eqref{eq:SuffCondDomKm} to ensure that a given function $g\colon\R_+\to\R$ lies in $\dom\Sigma^{m+1}$. We also note that establishing or refuting the converse inclusion constitutes an interesting open problem.

We now turn to a prominent example obtained by iterated application of the map $\Sigma$ to the logarithm function. We begin by recalling, via a proposition, the definition of the \emph{multiple gamma functions} when restricted to the domain $\R_+$. These functions form a sequence $m\mapsto G_m$ ($m\in\N$) in $\cC^{\infty}(\R_+)$, where $G_1=\Gamma$ is the Euler gamma function and $G_2=G$ is the Barnes $G$-function. For background, see, for example, Srivastava and Choi \cite[pp.\ 56--57]{SriCho12}.

\begin{proposition}[Multiple Gamma Functions]\label{prop:MulGamFuct4}
There exists an infinite sequence of functions $m\mapsto G_m$ ($m\in\N$) in $\cC^{\infty}(\R_+)$ that is uniquely determined by the following conditions:
\begin{itemize}
\item $G_0(x)=x$ for all $x>0$;
\item $G_{m+1}(x+1)=G_m(x){\,}G_{m+1}(x)$ for all $m\in\N$ and $x>0$;
\item $G_m(1)=1$ for all $m\in\N$;
\item The function $D^m(\ln\circ\, G_m)$ is increasing for all $m\in\N$.
\end{itemize}
\end{proposition}

The following result provides an explicit expression for repeated principal indefinite sums of the logarithm function in terms of multiple gamma functions; see also \cite[Section 12.1]{MarZen22}.

\begin{proposition}\label{prop:MulGamFuct5}
For any $m,n\in\N$, the function $\ln\circ\, G_m$ lies in
\begin{equation}\label{eq:IndMulrr0}
\cD^{m+1}\cap\cK^{m+n}_{(-1)^n}(\R_+)
\end{equation}
and we have
\begin{equation}\label{eq:IndMulrr1}
\Sigma_x\, \ln G_m(x) ~=~ \ln G_{m+1}(x)\qquad (x>0).
\end{equation}
In particular, the logarithm function lies in $\dom\Sigma^{m+1}$ and we have
$$
\Sigma^{m+1}_x\ln x ~=~ \ln G_{m+1}(x)\qquad (x>0).
$$
\end{proposition}

\begin{proof}
We proceed by induction on $m$. The initial case $m=0$ is straightforward. Indeed, we have
$$
D^{n+1}_x\ln x ~\in ~\cK^{-1}_{(-1)^n}(\R_+)
$$
and therefore, by Proposition~\ref{prop:ffprime32}, the function $\ln x$ lies in the set defined in \eqref{eq:IndMulrr0} with $m=0$. Moreover, identity \eqref{eq:IndMulrr1} also clearly holds when $m=0$ (cf.\ Bohr-Mollerup's theorem).

Now assume that the result holds for some $m\in\N$, and let us show that it also holds for $m+1$. On the one hand, by the definition of the multiple gamma functions, we have
$$
\Delta\ln G_{m+1}(x) ~=~ \ln\frac{G_{m+1}(x+1)}{G_{m+1}(x)} ~=~ \ln G_m(x)\qquad (x>0),
$$
which, by the induction hypothesis, lies in the set defined in \eqref{eq:IndMulrr0}. On the other hand, the final condition in Proposition~\ref{prop:MulGamFuct4}, together with Proposition~\ref{prop:ffprime32}, implies that the function $\ln\circ\, G_{m+1}$ lies in $\cK^{m+1}_1(\R_+)$. Combining this observation with the definition of the map $\Sigma$ yields identity \eqref{eq:IndMulrr1}.

Finally, since the function $\ln\circ\, G_m$ also lies in 
$$
\cD^{m+n+1}\cap\cK^{m+n+1}_{(-1)^{n+1}}(\R_+)
$$
for any $n\in\N$, it follows from \eqref{eq:IndMulrr1} that the function $\ln\circ\, G_{m+1}$ lies in
$$
\cD^{m+2}\cap\cK^{m+n+1}_{(-1)^n}(\R_+),
$$
which completes the induction. The final assertion of the proposition then follows immediately.
\end{proof}

\begin{remark}
From Proposition~\ref{prop:MulGamFuct5}, we observe that the function
$$
f ~=~ D^{m+1}(\ln\circ\, G_m)
$$
satisfies
$$
(-1)^n\, D^nf(x)~\geq ~0\qquad (n\in\N,~x>0),
$$
which precisely means that it is \emph{completely monotone}. By Bernstein's little theorem, it follows that the function $\ln\circ\, G_m$---and hence $G_m$ itself---is real analytic (for a recent reference, see \cite[Section 2]{LamMarZen25}). Moreover, for any $a>0$, the function $\ln\circ\, G_m$ admits a Newton series expansion at $a$ that converges on all of $\R_+$; see \cite[Theorem 3.1]{LamMarZen25}. This means that, for any $a,x>0$,
\begin{eqnarray*}
\ln G_m(x) &=& \sum_{k=0}^{\infty}\tchoose{x-a}{k}\,\Delta^k\ln G_m(a)\\
&=& \sum_{k=0}^m\tchoose{x-a}{k}\ln G_{m-k}(a)+\sum_{k=1}^{\infty}\tchoose{x-a}{k+m}\,\Delta^k\ln a.
\end{eqnarray*}
Moreover, the convergence of the series is uniform on compact subsets of $\R_+$. Setting $a=1$ for instance, we obtain
$$
\ln G_m(x) ~=~ \sum_{k=1}^{\infty}\tchoose{x-1}{k+m}\,(\Delta_t^k\ln t)\big|_{t=1}\qquad (x>0).
$$
In the special case $m=1$, we recover the following Newton series expansion of the log-gamma function:
$$
\ln\Gamma(x) ~=~ \sum_{k=1}^{\infty}\tchoose{x-1}{k+1}\,(\Delta_t^k\ln t)\big|_{t=1}\qquad (x>0)
$$
(see, for instance, Graham {\em et al.}~\cite[p.\ 192]{GraKnuPat94}).
\end{remark}

We conclude this section by noting that a function lying in $\dom\Sigma$ need not lie in $\dom\Sigma^2$. More precisely, the following strict inclusion holds:
$$
\dom\Sigma^2~\varsubsetneq ~\dom\Sigma.
$$
The next example provides a clear illustration of this fact.

\begin{example}
Let the function $g\colon\R_+\to\R$ be defined by
$$
g(x) ~=~ 2^{-\lfloor x\rfloor}(2-2\{x\}+\{x\}^2)\qquad (x>0),
$$
where $\{x\}=x-\lfloor x\rfloor$ denotes the fractional part of $x$ (see Figure~\ref{fig:PwPa}). For any $k\in\N$, the restriction of $g$ to the interval $(k,k+1]$ is the quadratic polynomial
$$
g(x) ~=~ 2^{-k}{\,}(2-2(x-k)+(x-k)^2)\qquad (x\in (k,k+1]).
$$
Hence, the graph of this restriction is a convex parabola with vertex at $x=k+1$. It follows readily that $g$ lies in $\cD^0\cap\cK^0_{-1}(\R_+)$, and hence also in $\dom\Sigma$. Now define a function $f\colon\R_+\to\R$ by $f(x)=2-2g(x)$. One easily verify that $f(1)=0$ and that $\Delta f(x)=g(x)$. Since $f$ lies in $\cK^0_1(\R_+)$, it follows that $f=\Sigma g$. However, this function also lies in $\cD^1\setminus\cD^0$, but not in $\cK^1$. Consequently, $\Sigma^2 g$ does not exist.
\end{example}

\begin{figure}[htbp]
\begin{center}
\includegraphics[height=.35\linewidth]{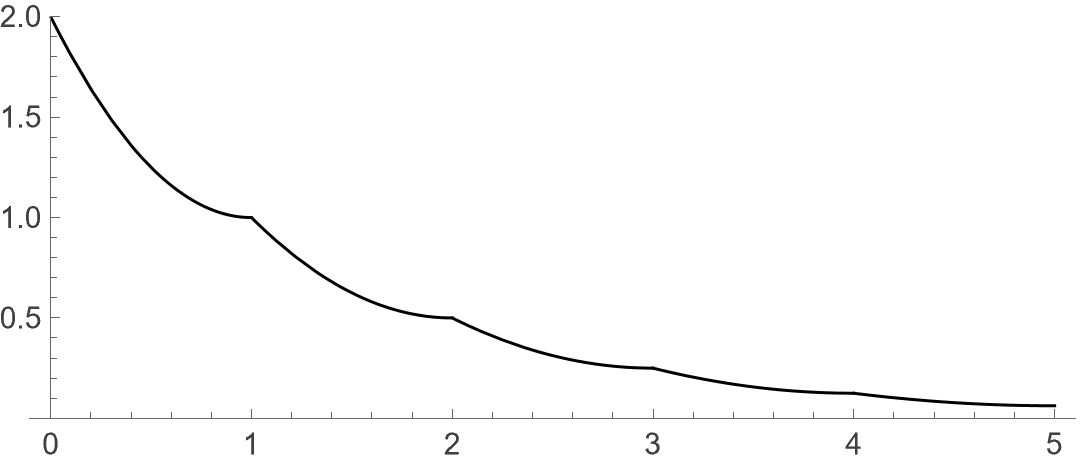}
\caption{Graph of $g(x)$ on $(0,5]$}
\label{fig:PwPa}
\end{center}
\end{figure}

\section{Iterating the Analog of Gauss' Limit}

Recall from \cite[p.\ 47]{MarZen22} that if $g$ lies in $\cD^p\cap\cK^p$ for some $p\in\N$, then for every $x>0$ and every $n\in\N^*{=\N\setminus\{0\}}$, the following identities hold:
\begin{eqnarray}
\Sigma g(n) &=& \sum_{k=1}^{n-1}g(k),\label{eq:Gau1}\\
\Sigma g(x+n) &=& \Sigma g(x)+\sum_{k=0}^{n-1}g(x+k),\label{eq:Gau2}
\end{eqnarray}
and
\begin{equation}\label{eq:Gau3}
\Sigma g(x) ~=~ f_n^p[g](x) + \rho_n^{p+1}[\Sigma g](x),
\end{equation}
where the function $f_n^p[g]\colon\R_+\to\R$ is defined in identity \eqref{eq:fnp2} and the remainder term $\rho_n^{p+1}[\Sigma g]\colon\R_+\to\R$ is given by
\begin{equation}\label{eq:Gau4}
\rho_n^{p+1}[\Sigma g](x) ~=~ \Sigma g(x+n)-\Sigma g(n)-\sum_{j=1}^p\tchoose{x}{j}\,\Delta^{j-1} g(n)\qquad (x>0).
\end{equation}
As discussed in the Introduction, the \emph{analog of Gauss' limit} is given by identity \eqref{eq:lim1}. In view of \eqref{eq:Gau3}, this identity is equivalent to the limit
$$
\lim_{n\to\infty}\,\rho_n^{p+1}[\Sigma g](x) ~=~ 0\qquad (x>0).
$$
It can further be shown \cite[Chapter 3]{MarZen22} that this latter convergence is uniform on any bounded subset of $\R_+$. When $g(x)=\ln x$ and $p=1$, we retrieve Gauss' limit \eqref{eq:GaussLimit00} for the gamma function. 

The aim of this section is to extend the analog of Gauss' limit to repeated principal indefinite sums, that is, to the function $\Sigma^{m+1}g$ for any $m\in\N$. This extension is stated in Corollary~\ref{lemma:GaussLim4659mco} below. We begin by recalling the discrete counterpart of Taylor's theorem and provide an elementary proof (for an alternative formulation, see, e.g., Lakshmikantham and Trigiante \cite[Theorem~1.1.2]{LakTri02}). We then establish a proposition that generalizes identities \eqref{eq:Gau1}--\eqref{eq:Gau3}.

\begin{lemma}[Discrete Counterpart of Taylor's Theorem]\label{lemma:Taylo68IntFReD}
For any $x>0$, any $m,n\in\N$, and any function $f\colon\R_+\to\R$, we have
$$
f(x+n) ~=~ \sum_{j=0}^m\tchoose{n}{j}\,\Delta^jf(x)
+\sum_{k=0}^{n-1}\tchoose{n-k-1}{m}\,\Delta^{m+1}f(x+k).
$$
\end{lemma}

\begin{proof}
We proceed by induction on $m$. The case when $m=0$ reduces to the following telescoping sum
$$
f(x+n) ~=~ f(x)+\sum_{k=0}^{n-1}\Delta f(x+k).
$$
Suppose now that the result holds for some $m\in\N$ and let us show that it still holds for $m+1$. Let $x>0$ and $n\in\N$. Using summation by parts, we obtain
\begin{eqnarray*}
\lefteqn{\sum_{k=0}^{n-1}\tchoose{n-1-k}{m}\,\Delta^{m+1}f(x+k) ~=~ \sum_{k=0}^{n-1}\big({-}\Delta_k\tchoose{n-k}{m+1}\big)\,\Delta^{m+1}f(x+k)}\\
&=& -\tchoose{n-k}{m+1}\,\Delta^{m+1}f(x+k)\Big|_0^n
+\sum_{k=0}^{n-1}\tchoose{n-k-1}{m+1}\,\Delta^{m+2}f(x+k)\\
&=& \tchoose{n}{m+1}\,\Delta^{m+1}f(x)
+\sum_{k=0}^{n-1}\tchoose{n-k-1}{m+1}\,\Delta^{m+2}f(x+k).
\end{eqnarray*}
Using the induction hypothesis, we now see that the result still holds for $m+1$.
\end{proof}

\begin{proposition}\label{lemma:GaussLim4659m}
For any $x>0$, any $m,p\in\N$, any $n\in\N^*$, and any function $g$ lying in $\cD^p\cap\cK^{p+m}$, we have
\begin{eqnarray}
\Sigma^{m+1} g(n) &=& \sum_{k=1}^{n-1}\tchoose{n-k-1}{m}{\,}g(k),\label{lemma:GaussLim46591m}\\
\Sigma^{m+1} g(x+n) &=& \sum_{j=0}^m\tchoose{n}{j}\,\Sigma^{m-j+1}g(x)
+\sum_{k=0}^{n-1}\tchoose{n-k-1}{m}{\,}g(x+k),\label{lemma:GaussLim46592m}
\end{eqnarray}
and
\begin{equation}\label{lemma:GaussLim46593m0}
\Sigma^{m+1}g(x) ~=~ f_n^{p+m}[\Sigma^m g](x)+\rho_n^{p+m+1}[\Sigma^{m+1}g](x),
\end{equation}
where
\begin{eqnarray}
f_n^{p+m}[\Sigma^m g](x) &=& \sum_{k=1}^{n-1}\tchoose{n-k-1}{m}{\,}g(k)-\sum_{k=0}^{n-1}\tchoose{n-k-1}{m}{\,}g(x+k)\nonumber\\
&& \null + \sum_{j=1}^{p+m}\tchoose{x}{j}\,\Delta^j\Sigma^{m+1}g(n)
-\sum_{j=1}^m\tchoose{n}{j}\,\Sigma^{m-j+1}g(x).\label{lemma:GaussLim46594m}
\end{eqnarray}
\end{proposition}

\begin{proof}
We have $g\in\dom\Sigma^{m+1}$ by Proposition~\ref{prop:SuffCondDomKm44}. Identity \eqref{lemma:GaussLim46592m} then follows immediately by applying Lemma~\ref{lemma:Taylo68IntFReD} to the function $f=\Sigma^{m+1} g$. Identity \eqref{lemma:GaussLim46591m} is obtained by setting $x=1$ in \eqref{lemma:GaussLim46592m}, while identity \eqref{lemma:GaussLim46593m0} follows from \eqref{eq:Gau3} applied to $\Sigma^m g$, with $p$ replaced by $p+m$.

We now establish \eqref{lemma:GaussLim46594m}. Using \eqref{lemma:GaussLim46593m0} and then applying \eqref{eq:Gau4} to $\Sigma^m g$, we obtain
\begin{multline*}
f_n^{p+m}[\Sigma^m g](x) ~=~ \Sigma^{m+1}g(x)-\rho_n^{p+m+1}[\Sigma^{m+1}g](x)\\
=~ \Sigma^{m+1}g(x)-\Sigma^{m+1} g(x+n)+\Sigma^{m+1} g(n)+\sum_{j=1}^{p+m}\tchoose{x}{j}\,\Delta^j\Sigma^{m+1} g(n).
\end{multline*}
Using \eqref{lemma:GaussLim46591m} and \eqref{lemma:GaussLim46592m}, this expression becomes
\begin{eqnarray*}
f_n^{p+m}[\Sigma^m g](x) &=& -\sum_{j=1}^m\tchoose{n}{j}\,\Sigma^{m-j+1}g(x)
-\sum_{k=0}^{n-1}\tchoose{n-k-1}{m}{\,}g(x+k)\\
&& +\sum_{k=1}^{n-1}\tchoose{n-k-1}{m}{\,}g(k)+\sum_{j=1}^{p+m}\tchoose{x}{j}\,\Delta^j\Sigma^{m+1} g(n),
\end{eqnarray*}
which establishes \eqref{lemma:GaussLim46594m}, and thus completes the proof of the proposition.
\end{proof}

\begin{corollary}[Iterated Form of the Analog of Gauss' Limit]\label{lemma:GaussLim4659mco}
For any $x>0$, any $m,p\in\N$, and any function $g$ lying in $\cD^p\cap\cK^{p+m}$, we have
\begin{equation}\label{lemma:GaussLim46593m}
\Sigma^{m+1} g(x) ~=~ \lim_{n\to\infty}f_n^{p+m}[\Sigma^m g](x).
\end{equation}
\end{corollary}

\begin{proof}
This identity follows immediately by applying \eqref{eq:lim1} to the function $\Sigma^m g$, which lies in $\cD^{p+m}\cap\cK^{p+m}$ (cf.\ Proposition \ref{prop:338-domSm}).
\end{proof}

We can now readily see that identities \eqref{eq:Gau1}--\eqref{eq:Gau3}, as well as \eqref{eq:lim1}--\eqref{eq:fnp2}, are recovered by setting $m=0$ in \eqref{lemma:GaussLim46591m}--\eqref{lemma:GaussLim46593m}. 

In the following corollary, we provide an identity that yield an alternative expression for $f_n^{p+m}[\Sigma^m g](x)$.

\begin{corollary}
For any $x>0$, any $m,p\in\N$, any $n\in\N^*$, and any function $g\colon\R_+\to\R$, we have
$$
\sum_{k=1}^{n-1}\tchoose{n-k-1}{m}{\,}g(k) + \sum_{j=1}^{p+m}\tchoose{x}{j}\,\Delta^j\Sigma^{m+1}g(n) ~=~ \sum_{k=1}^{n-1}\tchoose{x+n-k-1}{m}{\,}g(k) + \sum_{j=1}^p\tchoose{x}{m+j}\,\Delta^{j-1}g(n).
$$
\end{corollary}

\begin{proof}
Using \eqref{lemma:GaussLim46591m}, the left-hand side of the stated identity can be rewritten as
\begin{eqnarray*}
\sum_{j=0}^{p+m}\tchoose{x}{j}\,\Delta^j\Sigma^{m+1}g(n) &=& \sum_{j=0}^m\tchoose{x}{j}\,\Sigma^{m+1-j}g(n)+\sum_{j=m+1}^{p+m}\tchoose{x}{j}\,\Delta^{j-(m+1)}g(n)\\
&=& \sum_{j=0}^m\tchoose{x}{m-j}\,\Sigma^{j+1}g(n)+\sum_{j=1}^p\tchoose{x}{m+j}\,\Delta^{j-1}g(n).
\end{eqnarray*}
Again using \eqref{lemma:GaussLim46591m}, we obtain
\begin{eqnarray*}
\sum_{j=0}^m\tchoose{x}{m-j}\,\Sigma^{j+1}g(n) &=& \sum_{j=0}^m\tchoose{x}{m-j}\,\sum_{k=1}^{n-1}\tchoose{n-k-1}{j}g(k)\\
&=& \sum_{k=1}^{n-1} g(k)\,\sum_{j=0}^m\tchoose{x}{m-j}\tchoose{m-k-1}{j},
\end{eqnarray*}
where the inner sum reduces to $\tchoose{x+n-k-1}{m}$ by the classical Chu-Vandermonde's identity \cite[p.\ 170]{GraKnuPat94}. This establishes the stated identity.
\end{proof}

Interestingly, the iterated form of the analog of Gauss' limit, as given in \eqref{lemma:GaussLim46593m}, can yield rather surprising formulas. To demonstrate this, we first focus on the Barnes $G$-function and then consider the sequence of multiple gamma functions.

\begin{example}[The Barnes Function]\label{ex:324-Bar639}
Consider $g(x)=\ln x$, so that $\Sigma g(x)=\ln\Gamma(x)$ and $\Sigma^2g(x)=\ln G(x)$, and hence $p=1$ and $m=1$ in Proposition~\ref{lemma:GaussLim4659m}. In this case, we obtain
\begin{eqnarray*}
f_n^2[\Sigma g](x) &=& \sum_{k=1}^{n-1}(n-k-1)\ln k-\sum_{k=0}^{n-1}(n-k-1)\ln(x+k)\\
&& \null + x\ln\Gamma(n)+\tchoose{x}{2}\ln n-n\ln\Gamma(x).
\end{eqnarray*}
The limit in \eqref{lemma:GaussLim46593m} then becomes
$$
\ln G(x) ~=~ \lim_{n\to\infty}f_n^2[\Sigma g](x).
$$
In the multiplicative notation, this yields the following unexpected formula:
\begin{equation}\label{eq:325b-7re6fgsd}
G(x) ~=~ \lim_{n\to\infty}\frac{n!^x{\,}n^{{x\choose 2}}}{x^n\,\Gamma(x)^{n+1}}
\,\prod_{k=1}^n\Big(1+\frac{x}{k}\Big)^{k-n}.
\end{equation}
This expression is a remarkable variant of the following analog of Gauss' limit for the Barnes $G$-function, obtained directly by taking $g(x)=\ln\Gamma(x)$ and $p=2$ in \eqref{eq:lim1} and \eqref{eq:fnp2} (see \cite[p.\ 224]{MarZen22})
\begin{equation}\label{eq:325b-7re6fgsd1}
G(x) ~=~ \lim_{n\to\infty}\frac{n!^x{\,}n^{{x\choose 2}}}{\Gamma(x)}\,\prod_{k=1}^n\frac{\Gamma(k)}{\Gamma(x+k)}{\,}.
\end{equation}
The advantage of the representation \eqref{eq:325b-7re6fgsd} lies in the fact that the gamma function appears only outside the product.
\end{example}

\begin{example}[The Multiple Gamma Functions]
Example~\ref{ex:324-Bar639} can be easily generalized by considering the sequence of the multiple gamma functions as introduced in the previous section. In this case, \eqref{eq:325b-7re6fgsd} can be generalized into
$$
G_{m+1}(x) ~=~ \lim_{n\to\infty}\frac{n^{{x\choose m+1}}}{x^{{n\choose m}}}\,\Bigg(\prod_{j=1}^m\frac{G_{m+1-j}(n+1)^{{x\choose j}}}{G_{m+1-j}(x)^{{n+1\choose j}}}\Bigg)\prod_{k=1}^n\Big(1+\frac{x}{k}\Big)^{-{n-k\choose m}}.
$$
Furthermore, the generalization of \eqref{eq:325b-7re6fgsd1} is given by (see also \cite[Section 12.1]{MarZen22})
$$
G_{m+1}(x) ~=~ \lim_{n\to\infty}\frac{1}{G_m(x)}\,\prod_{j=1}^{m+1}G_{m+1-j}(n+1)^{{x\choose j}}\,\prod_{k=1}^n\frac{G_m(k)}{G_m(x+k)}{\,}.\qedhere
$$
\end{example}

We end this section with a closer examination of the values of the function $\Sigma^{m+1} g$ at the natural integers. First, by definition of the map $\Sigma$, we have
$$
\Sigma g(1) ~=~ 0
$$
for any $g\in\dom\Sigma$. More generally, it follows from \eqref{lemma:GaussLim46591m} that for any $m\in\N$ and any $g\in\dom\Sigma^{m+1}$, we have
\begin{equation}\label{lemma:GaussLim46597m}
\Sigma^{m+1}g(n) ~=~ 0\qquad (n=1,\ldots,m+1).
\end{equation}
Moreover,
for any $q\in\N$, we have
\begin{equation}\label{eq:23hk5j}
\Sigma^{m+1}g(m+1+q) ~=~ \sum_{k=1}^q\tchoose{m+q-k}{m}{\,}g(k).
\end{equation}
For instance,
\begin{eqnarray*}
\Sigma^{m+1}g(m+2) &=& g(1),\\
\Sigma^{m+1}g(m+3) &=& (m+1){\,}g(1)+g(2),\quad\text{etc.}
\end{eqnarray*}


\begin{remark}
We can easily establish the following identity, which provides an alternative formula for \eqref{eq:23hk5j}:
$$
\Sigma^{m+1} g(m+1+q) ~=~ \sum_{k=1}^q\tchoose{m+q}{m+k}\,\Delta^{k-1}g(1),\qquad n\in\N^*.
$$
Indeed, for any $n\in\N^*$, we may write
$$
\Sigma^{m+1} g(n) ~=~ \Sigma^{m+1}g(x+n-1)\big|_{x=1} ~=~ \sum_{k\geq 0}\tchoose{n-1}{k}\,\Delta^k\,\Sigma^{m+1}g(x)\big|_{x=1}{\,}.
$$
By using \eqref{lemma:GaussLim46597m}, the latter expression simplifies to:
$$
\sum_{k\geq m+1}\tchoose{n-1}{k}\,\Delta^k\,\Sigma^{m+1}g(1).
$$
Reindexing the sum and substituting $n=m+1+q$ yields the stated expression.
\end{remark}

\section{A Characterization}

In this section, we characterize the function $\Sigma^{m+1}g$ as a distinguished solution $f\colon\R_+\to\R$ to the equation $\Delta^{m+1}f=g$. The basic case $m=0$ simply reduces to a generalization of Bohr-Mollerup's theorem, as recalled in the Introduction. We restate this result in the next proposition. 

\begin{proposition}\label{prop:33-a7asd6f5Char}
Let $g$ lie in $\cD^p\cap\cK^p$ for some $p\in\N$ and let $f\colon \R_+\to\R$ be a
solution to the equation $\Delta f=g$. Then $f$ lies in $\cK^p$ if and only if $f = c + \Sigma g$ for some $c\in\R$. In particular, $\Sigma g$ is the unique solution to the equation $\Delta f=g$ that lies in $\cK^p$ and vanishes at $1$.
\end{proposition}

In the following proposition, we extend this result to characterize the function $\Sigma^{m+1}g$ for an arbitrary integer $m\in\N$. To formulate the assumptions on $g$, we naturally make use of Proposition~\ref{prop:SuffCondDomKm44}.

\begin{proposition}\label{prop:326-Ext-33z}
Let $g$ lie in $\cD^p\cap\cK^{p+m}$ for some $p,m\in\N$ and let $f\colon\R_+\to\R$ be a solution to the equation $\Delta^{m+1}f=g$. Then $\Delta^kf$ lies in $\cK^{p+m}$ for $k=0,1,\ldots,m$ if and only if
\begin{equation}\label{eq:PrPmSg4}
f ~=~ P_m+\Sigma^{m+1}g
\end{equation}
for some polynomial $P_m$ of degree at most $m$. In particular, $\Sigma^{m+1}g$ is the unique solution to the equation $\Delta^{m+1}f=g$ for which $\Delta^k f$ lies in $\cK^{p+m}$ and $f(k+1)=0$ for $k=0,1,\ldots,m$.
\end{proposition}

\begin{proof}
Let us establish the equivalence; the second statement will follow immediately.

(Sufficiency) Suppose that $f=P_m+\Sigma^{m+1}g$ for some polynomial $P_m$ of degree at most $m$. Then, for $k=0,1,\ldots,m$, the function
$$
\Delta^kf ~=~ \Delta^kP_m+\Sigma^{m+1-k}g
$$
necessarily lies in $\cK^{p+m}$. Indeed, this is clearly the case for the function $\Sigma^{m+1-k}g$. Moreover, the polynomial $\Delta^kP_m$ is of degree at most $m-k$ and therefore lies in both $\cK^{m-k+q}_1$ and $\cK^{m-k+q}_{-1}$ for all $q\in\N$.

(Necessity) Let us proceed by induction on $m\in\N$. By Proposition~\ref{prop:33-a7asd6f5Char}, the result holds for $m=0$. Suppose that it holds for some $m\in\N$ and let us show that it holds for $m+1$. Let $g\in\cD^p\cap\cK^{p+m+1}$ for some $p,m\in\N$, and let $f\colon\R_+\to\R$ be a solution to the equation $\Delta^{m+2}f=g$ such that $\Delta^k f$ lies in $\cK^{p+m+1}$ for $k=0,1,\ldots,m+1$. We only need to show that
$$
f ~=~ P_{m+1}+\Sigma^{m+2}g
$$
for some polynomial $P_{m+1}$ of degree at most $m+1$.

We first observe that the function $f_1=\Delta f$ satisfies $\Delta^{m+1}f_1=g$ and is such that $\Delta^kf_1$ lies in $\cK^{p+m+1}\subset\cK^{p+m}$ for $k=0,\ldots,m$. By the induction hypothesis, there exists a polynomial $P_m$ of degree at most $m$ such that
$$
\Delta f ~=~ f_1 ~=~ P_m+\Sigma^{m+1}g.
$$
Now, since
$$
P_m+\Sigma^{m+1}g ~\in ~\cD^{p+m+1}\cap\cK^{p+m+1}\qquad\text{and}\qquad f ~\in ~\cK^{p+m+1},
$$
by Proposition~\ref{prop:33-a7asd6f5Char} it follows that
$$
f ~=~ c+\Sigma(P_m+\Sigma^{m+1}g),
$$
for some $c\in\R$. Equivalently, we have (see \cite[Proposition 5.7]{MarZen22})
$$
f ~=~ P_{m+1}+\Sigma^{m+2}g
$$
for some polynomial $P_{m+1}$ of degree at most $m+1$. This completes the proof.
\end{proof}

It is noteworthy that the polynomial $P_m$ in Proposition~\ref{prop:326-Ext-33z} can be expressed in terms of the function $f$ as
\begin{equation}\label{eq:PolNew554}
P_m(x) ~=~ \sum_{j=0}^m\tchoose{x-1}{j}\,\Delta^jf(1).
\end{equation}
Indeed, using \eqref{lemma:GaussLim46597m} and \eqref{eq:PrPmSg4}, we see that $f(n)=P_m(n)$ for $n=1,\ldots,m+1$. Therefore, the Newton form of $P_m(x)$ at $x=1$ is
$$
P_m(x) ~=~ \sum_{j=0}^m\tchoose{x-1}{j}\,\Delta^jP_m(1) ~=~ \sum_{j=0}^m\tchoose{x-1}{j}\,\Delta^jf(1).
$$

We now derive an important corollary, which reminds us that the maps $\Delta$ and $\Sigma$ do not commute, and hence neither do $\Delta^{m+1}$ and $\Sigma^{m+1}$. First, we recall from \cite[p.\ 10]{MarZen22} that, for any $a>0$, any $p\in\N$, and any $f\colon\R_+\to\R$, the function $\rho_a^p[f]\colon (-a,\infty)\to\R$ is defined by
$$
\rho_a^p[f](x) ~=~ f(x+a)-\sum_{j=0}^{p-1}\tchoose{x}{j}\,\Delta^jf(a)\qquad (x>-a).
$$
Clearly, this definition is consistent with the identity in \eqref{eq:Gau4}.

\begin{corollary}\label{cor:326-Ext-33zc}
If a function $g$ lies in $\dom\Sigma^{m+1}$ for some $m\in\N$, then we have
$$
\Delta^{m+1}\Sigma^{m+1}g(x) ~=~ g(x)\qquad (x>0).
$$
If a function $f$ lies in $\cD^{p+m+1}\cap\cK^{p+2m+1}$ for some $m,p\in\N$, then we have
$$
\Sigma^{m+1}\Delta^{m+1}f(x) ~=~ \rho_1^{m+1}[f](x-1)\qquad (x>0).
$$
\end{corollary}

\begin{proof}
The first part is trivial. Now suppose that the function $f$ lies in $\cD^{p+m+1}\cap\cK^{p+2m+1}$ for some $m,p\in\N$. By Proposition~\ref{prop:ffprime324}, the function $g=\Delta^{m+1}f$ lies in $\cD^p\cap\cK^{p+m}$ and the function $\Delta^kf$ lies in $\cK^{p+m}$ for $k=0,1,\ldots,m$. Then, by Proposition~\ref{prop:326-Ext-33z} and the subsequent remark, we have
$$
\Sigma^{m+1}g ~=~ f-P_m\, ,
$$
where the polynomial $P_m$ is given in \eqref{eq:PolNew554}. Consequently,
$$
\Sigma^{m+1}g(x) ~=~ f(x-1+1)-\sum_{j=0}^m\tchoose{x-1}{j}\,\Delta^jf(1) ~=~ \rho_1^{m+1}[f](x-1).
$$
This completes the proof.
\end{proof}

Setting $m=0$ in Corollary~\ref{cor:326-Ext-33zc} immediately gives the following two special cases:
\begin{itemize}
\item\emph{If a function $g$ lies in $\dom\Sigma$, then we have
$$
\Delta\Sigma g(x) ~=~ g(x)\qquad (x>0).
$$
}
\item\emph{If a function $f$ lies in $\cD^p\cap\cK^p$ for some $p\in\N^*$, then we have
$$
\Sigma\Delta f(x) ~=~ f(x)-f(1)\qquad (x>0).
$$
}
\end{itemize}

\begin{remark}
Expanding the difference operator in the first identity of Corollary~\ref{cor:326-Ext-33zc} yields, for any $g\in\dom\Sigma^{m+1}$,
$$
g(x) ~=~ \sum_{k=0}^{m+1}\tchoose{m+1}{k}\, (-1)^{m+1-k}\,\Sigma^{m+1}g(x+k)\qquad (x>0).
$$
This identity provides a simple linear expression of $g$ in terms of $\Sigma^{m+1}g$.
\end{remark}

\begin{remark}
The second part of Corollary~\ref{cor:326-Ext-33zc} admits the following reformulation, which provides an interesting representation of the function $\rho_a^{m+1}[f]$:

\emph{If a function $f$ lies in $\cD^{p+m+1}\cap\cK^{p+2m+1}$ for some $m,p\in\N$, then we have}
$$
\big[\Sigma_t^{m+1}\Delta_t^{m+1}f(t+a-1)\Big]_{t=x+1} ~=~ \rho_a^{m+1}[f](x)\qquad (x>0).
$$
Indeed, defining $t\mapsto h(t)=f(t+a-1)$, we clearly have $h\in\cD^{p+m+1}\cap\cK^{p+2m+1}$. Applying the second part of Corollary~\ref{cor:326-Ext-33zc} to the function $h$ therefore yields
$$
\big[\Sigma_t^{m+1}\Delta_t^{m+1}h(t)\big]_{t=x+1} ~=~ \rho_1^{m+1}[h](x) ~=~ \rho_a^{m+1}[f](x)\qquad (x>0),
$$
which completes the proof.
\end{remark}

We now illustrate the main preceding results using the polygamma function
$$
\psi_{-2}(x) ~=~ \int_0^x\ln\Gamma(t)\, dt ~=~ \int_0^x\ln G_1(t)\, dt\qquad (x>0),
$$
when restricted to the set $\R_+$. Note that this function has recently been studied in depth in connection with the generalization of the additive form of the Bohr-Mollerup theorem (see the tutorial paper \cite{MarZen24}).

\begin{example}[The Polygamma Function $\psi_{-2}$]\label{ex:324-Poly2G639}
Consider the function $f_1\colon\R_+\to\R$ defined by the equation
$$
f_1(x) ~=~ \psi_{-2}(x)-\psi_{-2}(1)\qquad (x>0),
$$
or equivalently (see \cite{MarZen24}),
$$
f_1(x) ~=~ \Sigma_x(x\ln x-x+\psi_{-2}(1)),
$$
where the function $x\mapsto x\ln x-x+\psi_{-2}(1)$ lies in $\cD^2\cap\cK^{\infty}$ and
$$
\psi_{-2}(1) ~=~ \frac{1}{2}\ln(2\pi).
$$
We note that $f_1(1)=0$ while $f_1(2)\neq 0$. By \eqref{lemma:GaussLim46597m}, there is therefore no function $g\colon\R_+\to\R$ such that $f_1=\Sigma^2g$. However, this issue can be addressed by replacing $f_1$ with the function $f=\Sigma^2\Delta^2f_1$. Applying Corollary~\ref{cor:326-Ext-33zc} with $m=p=1$, we then obtain
\begin{eqnarray*}
f(x) &=& f_1(x)-f_1(1)-(x-1)\,\Delta f_1(1)\\
&=& \psi_{-2}(x)-1+(1-\psi_{-2}(1)){\,}x\qquad (x>0).
\end{eqnarray*}
Thus $f(x)=\Sigma^2g(x)$, where the function $g\colon\R_+\to\R$ is given by
$$
g(x) ~=~ \Delta^2 f_1(x) ~=~ -1+\Delta_x(x\ln x),
$$
and lies in $\cD^1\cap\cK^{\infty}$.
\end{example}

\section{A Generalization of Malmst\'en's Formula}

The following identities yield well-known integral representations of both the logarithm function and the log-gamma function on $\R_+$: 
\begin{eqnarray*}
\ln x &=& \int_0^{\infty}\big(1-e^{(1-x)t}\big)\,\frac{e^{-t}}{t}{\,}dt,\\
\ln\Gamma(x) &=& \int_0^{\infty}\Big((x-1)-\frac{e^{(1-x)t}-1}{e^{-t}-1}\Big)\,\frac{e^{-t}}{t}{\,}dt.
\end{eqnarray*}
The latter identity is commonly referred to as the \emph{Malmst\'en's formula}. For background, see for instance Erd\'elyi {\em et al.}~\cite[pp.\ 16, 20--21]{ErdMagOberTri81}.

Following a suggestion of N.~Zena\"idi (University of Li\`ege, \emph{Personal communication}, 2025), we now show how Malmst\'en's formula can be generalized by providing, for every $m\in\N$, an integral representation of the function 
$$
\ln G_{m+1}(x) ~=~ \Sigma_x^{m+1}\ln x\qquad (x>0).
$$
This generalization is stated in the following theorem. As an illustration, we derive the following integral representation of the logarithm of the Barnes function:
$$
\ln G(x) ~=~ \ln G_2(x) ~=~ \int_0^{\infty}\Big(\tchoose{x-1}{2}-\frac{e^{(1-x)t}-1}{(e^{-t}-1)^2}
+\frac{x-1}{e^{-t}-1}\Big)\,\frac{e^{-t}}{t}{\,}dt.
$$

\begin{theorem}[Generalization of Malmst\'en's Formula]
For any $m\in\N$, the function $\ln\circ{\,}G_{m+1}$ admits the following integral representation on $\R_+$
\begin{equation}\label{eq:329-logGmIntRe7}
\ln G_{m+1}(x) ~=~ \int_0^{\infty}\Sigma_x^{m+1}\big(1-e^{(1-x)t}\big)\,\frac{e^{-t}}{t}{\,}dt\qquad (x>0),
\end{equation}
where, for any $t>0$,
\begin{equation}\label{eq:329-logGmIntRe7bi}
\Sigma_x^{m+1}(1-e^{(1-x)t}) ~=~ \tchoose{x-1}{m+1}-\frac{e^{(1-x)t}-1}{(e^{-t}-1)^{m+1}}
+\sum_{j=1}^m\frac{\tchoose{x-1}{j}}{(e^{-t}-1)^{m+1-j}}{\,}.
\end{equation}
\end{theorem}

\begin{proof}
We first observe that, for any $t>0$, the function $x\mapsto e^{(1-x)t}$ lies in $\cD^0\cap\cK^{\infty}$ and that
$$
1-e^{(1-x)t} ~=~ \Delta_x^{m+1}\tchoose{x-1}{m+1}-\Delta_x^{m+1}\frac{e^{(1-x)t}}{(e^{-t}-1)^{m+1}}{\,}.
$$
Using Corollary~\ref{cor:326-Ext-33zc} with $p=1$ and
$$
f(x) ~=~ \tchoose{x-1}{m+1}-\frac{e^{(1-x)t}}{(e^{-t}-1)^{m+1}}{\,},
$$
we can easily derive the identity in \eqref{eq:329-logGmIntRe7bi}.

The integral representation \eqref{eq:329-logGmIntRe7} then follows from an application of the second part of Proposition~\ref{prop:326-Ext-33z} to the function $g(x)=\ln x$, which lies in $\cD^1\cap\cK^{\infty}$. Indeed, let $I_{m+1}(x)$ denote the right-hand integral in \eqref{eq:329-logGmIntRe7}. We clearly have
$$
\Delta^{m+1}I_{m+1}(x) ~=~ \int_0^{\infty}\big(1-e^{(1-x)t}\big)\,\frac{e^{-t}}{t}{\,}dt ~=~ \ln x\qquad (x>0),
$$
and, cf.\ \eqref{lemma:GaussLim46597m},
$$
I_{m+1}(1) ~=~ \cdots ~=~ I_{m+1}(m+1) ~=~ 0.
$$
Note that in this computation all integrals involved are well defined.

To conclude the proof, it suffices to show that, for $k=0,\ldots,m$, the function $\Delta^kI_{m+1}$ lies in $\cK^{m+1}$, or equivalently, that the function $D^{m+2}\Delta^kI_{m+1}$ eventually has a constant sign. However, for every $t>0$, the function
$$
D^{m+2}_x\,\Sigma_x^{m+1-k}\big(1-e^{(1-x)t}\big) ~=~ -D^{m+2}_x\,\frac{e^{(1-x)t}}{(e^{-t}-1)^{m+1-k}} ~=~ -\frac{(-t)^{m+2}\, e^{(1-x)t}}{(e^{-t}-1)^{m+1-k}}
$$
clearly has a constant sign, and hence so does the function $D^{m+2}\Delta^kI_{m+1}$ (here we may differentiate under the integral sign). This completes the proof.
\end{proof}

\section{Further Examples}

Beyond the multiple gamma functions, there are numerous examples that illustrate our results on repeated principal indefinite sums. In this section, we explore several of these examples.

\subsection{Polynomial functions}

We begin by considering the case in which the function $g$ is a polynomial. As explained in Graham \emph{et al}.\ \cite[p.\ 189]{GraKnuPat94}, in this discrete setting it is often more convenient to express polynomials in their \emph{Newton forms}, that is, as linear combinations of binomial coefficients. More precisely, any polynomial $P_q$ of degree at most $q\in\N$ admits the representation
$$
P_q(x) ~=~ \sum_{k=0}^q\tchoose{x}{k}\,\Delta^kP_q(0).
$$

Now, since such a polynomial $P_q$ lies in both $\cK^q_1$ and $\cK^q_{-1}$, it follows that, for any $m\in\N$, we may write (cf.\ \cite[Proposition 5.7]{MarZen22})
\begin{equation}\label{eq:pol446}
\Sigma^{m+1} P_q(x) ~=~ \sum_{k=0}^q\big(\Sigma^{m+1}_x\tchoose{x}{k}\big)\,\Delta^kP_q(0)\qquad (x>0).
\end{equation}
To compute the right-hand side of \eqref{eq:pol446}, it therefore suffices to derive an explicit expression for $\Sigma^{m+1}_x\tchoose{x}{k}$. The following proposition provides such an expression. Its proof follows straightforwardly from Propositions~\ref{prop:SuffCondDomKm44} and \ref{prop:33-a7asd6f5Char}.

\begin{proposition}\label{prop:pol2}
For any $k,m\in\N$, the function $g_k(x)=\tchoose{x}{k}$ lies in $\dom\Sigma^{m+1}$ and we have
$$
\Sigma^{m+1} g_k(x) ~=~ \tchoose{x-0^k}{k+m+1} ~=~ 
\begin{cases}
{\tchoose{x-1}{m+1}}_{\mathstrut}{\,}, & \text{if $k=0$},\\
\tchoose{x}{k+m+1}{\,}, & \text{if $k\geq 1$}.
\end{cases}
$$
Moreover, the function $\Sigma^{m+1}g_k$ lies in $\cD^{k+m+2}\cap\cK^{\infty}$.
\end{proposition}


Now, in light of Proposition~\ref{prop:pol2}, the identity in \eqref{eq:pol446} can be rewritten more explicitly as
$$
\Sigma^{m+1} P_q(x) ~=~ \tchoose{x-1}{m+1}\, P_q(0)+\sum_{k=1}^q\tchoose{x}{k+m+1}\,\Delta^kP_q(0)\qquad (x>0).
$$
It follows immediately that if the polynomial $P_q$ is of degree $q$, then $\Sigma^{m+1} P_q$ is a polynomial of degree $q+m+1$. Moreover, its coefficients can be easily computed using the standard conversion formulas between ordinary powers and binomial coefficients (see, e.g., \cite[p.\ 264]{GraKnuPat94}).

\subsection{The Reciprocal Function}

After discussing polynomials, we now turn to certain special rational functions. The first function we consider is the reciprocal function $g(x)=1/x$. As the following proposition shows, its repeated principal indefinite sum involves the \emph{digamma} function $\psi\colon\R_+\to\R$, which admits the following series representation
$$
\psi(x) ~=~ -\gamma+\sum_{k=1}^{\infty}\Big(\frac{1}{k}-\frac{1}{x+k-1}\Big)\qquad (x>0),
$$
where $\gamma$ denotes Euler's constant. For background on the digamma function, see for instance \cite{MarZen22,SriCho12}.

\begin{proposition}\label{prop:RecF32}
For any $m\in\N$, the reciprocal function $g(x)=1/x$ lies in $\dom\Sigma^{m+1}$ and we have
$$
\Sigma^{m+1}g(x) ~=~ \tchoose{x-1}{m}\, (\psi(x)-\psi(m+1))\qquad (x>0).
$$
Moreover, the function $\Sigma^{m+1}g$ lies in $\cD^{m+1}\cap\cK^{\infty}$.
\end{proposition}

\begin{proof}
Since the function $g$ lies in $\cD^0\cap\cK^{\infty}$, Proposition~\ref{prop:SuffCondDomKm44} implies that $\Sigma^{m+1}g$ exists and lies in $\cD^{m+1}\cap\cK^{\infty}$. We now prove the stated identity. The case $m=0$ is well known \cite[Section 10.2]{MarZen22}, and we may therefore assume that $m\geq 1$. For any $k\in\N$, define the function $f_k\colon\R_+\to\R$ by
$$
f_k(x) ~=~ \tchoose{x-1}{k}\, (\psi(x)-\psi(k+1))\qquad (x>0).
$$
Using the classical product rule for the difference operator, it is straightforward to verify that $\Delta f_{k+1}=f_k$. Moreover, we have the following claim, which establishes an important convexity property of $f_k$.


\begin{claim}
For any $k\in\N$, the function $f_k$ lies in $\cK^k_1(\R_+)$.
\end{claim}

{
\renewcommand{\qedsymbol}{$\lozenge$}
\begin{proof}[Proof of the Claim]
For any $j,k\in\N$ with $j\leq k$, define the function $g_{j,k}\colon\R_+\to\R$ by
$$
g_{j,k}(x) ~=~ (-1)^{k+j}\, x^j\,\psi(x)\qquad (x>0).
$$
Using the series representation of $\psi(x)$, we obtain
$$
g_{j,k}^{(k+1)}(x) ~=~ (-1)^{k+j+1}\,\sum_{n=0}^{\infty}D_x^{k+1}\frac{x^j}{x+n} ~=~ (j+1)!\,\sum_{n=0}^{\infty}\frac{n^j}{(x+n)^{j+2}}{\,}.
$$
Since the latter expression is positive throughout $\R_+$, it follows from Proposition~\ref{prop:ffprime32} that the function $g_{j,k}$ lies in $\cK^k_1(\R_+)$.

Now, by expanding the binomial coefficient in the definition of $f_k$, we see that this function can be expressed in the form
$$
f_k(x) ~=~ \sum_{j=0}^kc_j\, g_{j,k}(x) + P_k(x)\qquad (x>0),
$$
where $c_0,\ldots,c_k\geq 0$ and $P_k$ is a polynomial of degree at most $k$. Consequently, each summand on the right-hand side lies in $\cK^k_1(\R_+)$, and therefore so does the function $f_k$ itself.
\end{proof}
}

To complete the proof, it suffices to show that the identity $\Sigma f_k=f_{k+1}$ holds for all $k\in\N$. We proceed by induction on $k$. The case $k=0$ is straightforward. Indeed, we have
$$
f_0 ~=~ \Sigma g ~\in ~\cD^1\cap\cK^{\infty} ~\subset ~ \cD^1\cap\cK^1,
$$
and by the Claim, $f_1\in\cK^1$. Consequently, Proposition~\ref{prop:33-a7asd6f5Char} implies that $\Sigma f_0=f_1$. Now suppose that the statement holds for some $k\in\N$ and let us show that it also holds for $k+1$. We have
$$
f_{k+1} ~=~ \Sigma f_k ~=~ \Sigma^{k+2}g ~\in ~\cD^{k+2}\cap\cK^{\infty} ~\subset ~\cD^{k+2}\cap\cK^{k+2},
$$
and by the Claim, $f_{k+2}\in\cK^{k+2}$. Therefore, $\Sigma f_{k+1}=f_{k+2}$, and the induction is complete.
\end{proof}

\begin{remark}
Interestingly, from Propositions~\ref{prop:pol2} and \ref{prop:RecF32}, we can easily derive the following identity
$$
\Sigma^{m+1}\psi(x) ~=~ \tchoose{x-1}{m+1}\,\big(\psi(x)-\psi(m+2)+\psi(1)\big)\qquad (x>0).
$$
Indeed, to obtain this identity, it suffices essentially to apply the map $\Sigma^{m+1}$ to the representation $\psi(x)=\psi(1)+\Sigma_x\frac{1}{x}$.
\end{remark}

\subsection{Powers of the Reciprocal Function}

We now continue our study of repeated principal indefinite sums of rational functions by considering the powers of the reciprocal function. In this regard, we note that, for any $r\in\N$,
\begin{equation}\label{eq:63Fr}
\frac{1}{x^{r+1}} ~=~ \frac{(-1)^r}{r!}\, D_x^r\frac{1}{x}\qquad (x>0). 
\end{equation}
This identity suggests that it is convenient to derive an explicit formula for the function $\Sigma^{m+1}g^{(r)}$ in terms of $D^r\Sigma^{m+1}g$.

In this context, the following proposition is particularly significant, as it allows us to generate new examples from existing ones. In fact, it is simply a straightforward generalization of the second part of Corollary~\ref{cor:326-Ext-33zc}.

\begin{proposition}\label{prop:GenLinOp3}
Let $T$ be a linear operator on the space of functions from $\R_+$ to $\R$ that commutes with $\Delta$, that is, $T\Delta =\Delta T$, and let $m,p\in\N$. Suppose that a function $f\in\dom\Sigma^{m+1}$ satisfies
$$
Tf\in\cD^p\cap\cK^{p+m}\quad\text{and}\quad T\Sigma^{m+1}f\in\cK^{p+2m}.
$$
Then
$$
\Sigma^{m+1}T f(x) ~=~ \rho_1^{m+1}[T\Sigma^{m+1}f](x-1)\qquad (x>0).
$$
\end{proposition}

\begin{proof}
Setting $g=Tf$ and $\bar f=T\Sigma^{m+1}f$, we have $\Delta^{m+1}\bar f=g$ and $\Delta^k\bar f\in\cK^{p+m}$ for $k=0,1,\ldots,m$. By Proposition~\ref{prop:326-Ext-33z} and the subsequent remark, it follows that
$$
\Sigma^{m+1}g ~=~\bar f-P_m,
$$
where $P_m$ is the polynomial defined by
$$
P_m(x) ~=~ \sum_{j=0}^m\tchoose{x-1}{j}\,\Delta^j\bar f(1).
$$
Consequently,
\begin{eqnarray*}
\Sigma^{m+1}T f(x) ~=~ \Sigma^{m+1}g(x)  &=& T\Sigma^{m+1}f(x)-\sum_{j=0}^m\tchoose{x-1}{j}\,\Delta^j T\Sigma^{m+1}f(1)\\
&=& \rho_1^{m+1}[T\Sigma^{m+1}f](x-1)\qquad (x>0),
\end{eqnarray*}
as claimed.
\end{proof}

%

As discussed above, Proposition~\ref{prop:GenLinOp3} is particularly relevant when $T$ is the standard derivative $D$ or one of its higher-order iterates. In this connection, recall from \cite[Section 7.1]{MarZen22} that if
$$
g ~\in ~ \cC^r(\R_+)\cap\cD^p\cap\cK^{\max\{p,r\}}
$$
for some $p,r\in\N$, then
$$
\Sigma g ~\in ~ \cC^r(\R_+)\cap\cD^{p+1}\cap\cK^{\max\{p,r\}}.
$$
Combining this fact with Proposition~\ref{prop:GenLinOp3} applied to $T=D^r$, we immediately derive the following result.

\begin{proposition}\label{prop:TD334}
If $f$ lies in $\cC^r(\R_+)\cap\cD^{p+r}\cap\cK^{p+2m+r}$ for some $m,p,r\in\N$, then
$$
\Sigma^{m+1}f^{(r)}(x) ~=~ \rho_1^{m+1}[D^r\Sigma^{m+1}f](x-1)\qquad (x>0).
$$
\end{proposition} 

Now apply Proposition~\ref{prop:TD334}, for instance, to the function $f(x)=-1/x$, which corresponds to the choice $r=1$ in \eqref{eq:63Fr}. Using the identity (see, e.g., \cite[p.\ 352]{Gel71})
$$
D_x\tchoose{x-1}{m} ~=~ \tchoose{x-1}{m}\,(\psi(x)-\psi(x-m))\qquad (x\neq 1,2,\ldots,m),
$$
a straightforward but somewhat lengthy computation yields
\begin{eqnarray*}
\Sigma^{m+1}_x\frac{1}{x^2} &=& -\tchoose{x-1}{m}\,\big((\psi(x)-\psi(x-m))(\psi(x)-\psi(m+1))+\psi'(x)-\psi'(1)\big)\\
&& \null + \sum_{j=1}^m\tchoose{x-1}{m-j}\Big(\frac{(-1)^{j+1}}{j}(\psi(1)-\psi(j+1)\Big).
\end{eqnarray*}
The details of the computation are omitted.

Returning to the study of the multiple gamma functions, we observe that a judicious combination of Propositions~\ref{prop:MulGamFuct5}, \ref{prop:RecF32}, and \ref{prop:TD334} yields the following somewhat surprising result, which provides explicit expressions for their first derivatives.

\begin{corollary}
For any $m\in\N$ and any $x>0$, we have
$$
G'_{m+1}(x) ~=~ G_{m+1}(x)\,\Big(\tchoose{x-1}{m}(\psi(x)-\psi(m+1))+\sum_{j=0}^m\tchoose{x-1}{j}\, G'_{m+1-j}(1)\Big).
$$
\end{corollary}

\begin{proof}
Using Propositions~\ref{prop:MulGamFuct5}, \ref{prop:RecF32}, and \ref{prop:TD334}---after readily verifying that their respective assumptions are satisfied---we obtain, for $m\in\N$ and $x>0$,
\begin{eqnarray*}
\tchoose{x-1}{m}\, (\psi(x)-\psi(m+1)) &=& \Sigma^{m+1}_xD_x\ln x\\
&=& D_x\Sigma^{m+1}_x\ln x-\sum_{j=0}^m\tchoose{x-1}{j}\,\big[\Delta_z^jD_z\,\Sigma^{m+1}_z\ln z\big]_{z=1}\\
&=& D_x\ln G_{m+1}(x)-\sum_{j=0}^m\tchoose{x-1}{j}\, D\ln G_{m+1-j}(1).
\end{eqnarray*}
This completes the proof.
\end{proof}

\subsection{Negative Falling Factorials}

We conclude this section by examining the repeated principal indefinite sum of the function $g_n\colon\R_+\to\R$, defined for each $n\in\N$ by
$$
g_n(x) ~=~ (-1)^n\,\Delta^n_x\,\frac{1}{x} ~=~ \frac{n!}{x(x+1)\,\cdots\, (x+n)}\qquad (x>0).
$$
This function admits a simple representation in terms of a \emph{negative falling factorial power} (see \cite[p.\ 188]{GraKnuPat94})
$$
g_n(x) ~=~ (-1)^n\, n!\, (x-1)^{\underline{-n-1}} \qquad (x>0).
$$
The very definition of $g_n$ suggests applying Proposition~\ref{prop:GenLinOp3} with $T=\Delta^n$. However, since $g_n$ also possesses the partial fraction expansion
$$
g_n(x) ~=~ \sum_{k=0}^n\tchoose{n}{k}\,\frac{(-1)^k}{x+k}\qquad (x>0),
$$
it follows that, for any $m\in\N$, we may write
$$
\Sigma^{m+1}g_n(x) ~=~ \sum_{k=0}^n\tchoose{n}{k}\, (-1)^k\,\Sigma^{m+1}_x E_x^k\,\frac{1}{x}{\,},
$$
where $E$ is the classical \emph{shift} operator defined by $Ef(x)=f(x+1)$. Now, applying Proposition~\ref{prop:GenLinOp3} with $T=E^k$, followed by Proposition~\ref{prop:RecF32}, we obtain
\begin{multline*}
\Sigma^{m+1}_xE_x^k\,\frac{1}{x} ~=~ E_x^k\,\Sigma^{m+1}_x\,\frac{1}{x}-\sum_{j=0}^m\tchoose{x-1}{j}\,\Big[\Sigma_z^{m+1-j}\,\frac{1}{z}\Big]_{z=k+1}\\
=~ \tchoose{x+k-1}{m}(\psi(x+k)-\psi(m+1))-\sum_{j=0}^m\tchoose{x-1}{j}\tchoose{k}{m-j}(\psi(k+1)-\psi(m+1-j)).
\end{multline*}
In the special case $m=0$, this formula simplifies to
$$
\Sigma g_n(x) ~=~ \sum_{k=0}^n\tchoose{n}{k}\, (-1)^k\, (\psi(x+k)-\psi(k+1))\qquad (x>0).
$$

\section{Discrete Analog of Cauchy's Formula for Repeated Integration}

We begin this final section by recalling the well-known Taylor theorem with integral remainder, as it applies to real-valued functions defined on $\R_+$. This result can be proved in a straightforward manner by induction, using repeated integration by parts.

\begin{theorem}[Taylor's Theorem]\label{thm:71Tay32}
Let $a>0$ and let $f\in\cC^{m+1}(\R_+)$ for some $m\in\N$. Then, for any $x\geq 0$, we have
$$
f(x+a) ~=~ \sum_{j=0}^m\frac{x^j}{j!}\, f^{(j)}(a)+\int_0^x\frac{(x-t)^m}{m!}\, f^{(m+1)}(t+a)\, dt.
$$
\end{theorem}

We now offer a discrete analog of this theorem, in which the integral remainder is naturally replaced with a principal indefinite sum. This result---which may be regarded as a generalization of Lemma~\ref{lemma:Taylo68IntFReD}---is stated in the following theorem, and its proof likewise proceeds by induction, relying on repeated summation by parts. This approach, however, requires several assumptions ensuring the existence of the principal indefinite sums involved.

\begin{theorem}[Discrete Counterpart of Taylor's Theorem]\label{thm:72DiscTay33}
Let $a>0$, $m\in\N$, and let $f\colon\R_+\to\R$ be a function satisfying the following assumptions for $k=0,\ldots,m$:
\begin{itemize}
\item the function $t\mapsto\tchoose{x-t}{k}\,\Delta^{k+1}f(t+a-1)$ lies in $\dom\Sigma$;
\item the function $t\mapsto\tchoose{x-t+1}{k}\,\Delta^kf(t+a-1)$ lies in $\cD^{p_k}\cap\cK^{p_k}$ for some $p_k\in\N^*$.
\end{itemize}
Then, for any $x\geq 0$, we have
$$
f(x+a) ~=~ \sum_{j=0}^m\tchoose{x}{j}\,\Delta^jf(a)+\big[\Sigma_t\,\tchoose{x-t}{m}\,\Delta^{m+1}f(t+a-1)\big]_{t=x+1}.
$$
\end{theorem}

\begin{proof}
We proceed by induction on $m$. The result holds for $m=0$. Indeed, using the second assumption on $f$ (with $k=0$), we obtain
\begin{eqnarray*}
\big[\Sigma_t\,\Delta f(t+a-1)\big]_{t=x+1} &=& [f(t+a-1)-f(a)]_{t=x+1}\\
&=& f(x+a)-f(a).
\end{eqnarray*}
Suppose now that the result holds for some $m\in\N$, and let us show that it still holds for $m+1$. Using the first assumption on $f$, we obtain
\begin{multline*}
\Sigma_t\,\tchoose{x-t}{m}\,\Delta^{m+1}f(t+a-1) ~=~ \Sigma_t\,\big({-}\Delta_t\tchoose{x-t+1}{m+1}\big)\,\Delta^{m+1}f(t+a-1)\\
=~ \Sigma_t\,\Big(\Delta_t\Big({-}\tchoose{x-t+1}{m+1}\,\Delta^{m+1}f(t+a-1)\Big)+\tchoose{x-t}{m+1}\,\Delta^{m+2}f(t+a-1)\Big).
\end{multline*}
Using the second assumption on $f$, this expression becomes
$$
{-}\tchoose{x-t+1}{m+1}\,\Delta^{m+1}f(t+a-1)+\tchoose{x}{m+1}\,\Delta^{m+1}f(a)+\Sigma_t\,\tchoose{x-t}{m+1}\,\Delta^{m+2}f(t+a-1).
$$
Evaluating at $t=x+1$, we obtain
$$
\tchoose{x}{m+1}\,\Delta^{m+1}f(a)+\big[\Sigma_t\,\tchoose{x-t}{m+1}\,\Delta^{m+2}f(t+a-1)\big]_{t=x+1}\, .
$$
Using the induction hypothesis, we conclude that the result holds for $m+1$, which completes the proof.
\end{proof}

Theorem~\ref{thm:72DiscTay33} is remarkable in that it provides an explicit expression for the remainder term in the Newton expansion at $x=a$ of the function $f(x+a)$, expressed in terms of a principal ``definite'' sum. This is entirely analogous to Theorem~\ref{thm:71Tay32}, where the remainder term in the Taylor expansion is expressed as an integral.

Actually, Theorem~\ref{thm:72DiscTay33} is entirely new and likely holds significant potential for applications. Simplifying its assumptions, however, remains an interesting open problem.

We conclude this paper with a remarkable identity that reduces a repeated principal indefinite sum to a single principal ``definite'' sum.

\begin{corollary}\label{cor:72DiscTay34}
Let $m\in\N$, let $g\in\dom\Sigma^{m+1}$, and suppose that the function $f=\Sigma^{m+1}g$ satisfies the assumptions of Theorem~\ref{thm:72DiscTay33}. Then, we have
$$
\Sigma^{m+1}g(x) ~=~ \big[\Sigma_t\,\tchoose{x-t-1}{m}\, g(t)\big]_{t=x}\qquad (x>0).
$$
\end{corollary}

\begin{proof}
The result is immediate by applying Theorem~\ref{thm:72DiscTay33} (assuming $x>0$) to the function $f=\Sigma^{m+1}g$ at the value $a=1$. Explicitly,
$$
\Sigma^{m+1}g(x) ~=~ \sum_{j=0}^m\tchoose{x-1}{j}\,\Sigma^{m+1-j}g(1)+\big[\Sigma_t\,\tchoose{x-t-1}{m}\, g(t)\big]_{t=x}\qquad (x>0).
$$
This completes the proof.
\end{proof}

\begin{example}
Applying Corollary~\ref{cor:72DiscTay34} to $g(x)=\ln x$ with $m=1$, we obtain
\begin{eqnarray*}
\ln G(x) &=& \Sigma_x^2\ln x ~=~ \big[\Sigma_t\, (x-1-t)\ln t\big]_{t=x}\\
&=& (x-1)\ln\Gamma(x)-\Sigma_x (x\ln x) ~=~ (x-1)\ln\Gamma(x)-\ln K(x),
\end{eqnarray*}
where $K\colon\R_+\to\R$ is the $K$-function defined by (see, e.g., \cite[Section 12.5]{MarZen22})
$$
\ln K(x) ~=~ \Sigma_x (x\ln x)\qquad (x>0).
$$
Note that the function $x\mapsto x\ln x$ lies in $\cD^2\cap\cK^{\infty}$.
\end{example}

The ``continuous'' analog of Corollary~\ref{cor:72DiscTay34} is the well-known \emph{Cauchy formula for repeated integration}, which can be described as follows. Let $D^{-1}$ denote the operator from $\cC^0(\R_+)$ to $\cC^1(\R_+)$ defined by
$$
D^{-1}g(x) ~=~ \int_1^xg(t)\, dt.
$$
It is straightforward to verify that for $g\in\cC^0(\R_+)$ and $f\in\cC^1(\R_+)$,
$$
DD^{-1}g ~=~ g\qquad\text{and}\qquad D^{-1}Df ~=~ f-f(1),
$$
in complete analogy with the case $m=0$ of Corollary~\ref{cor:326-Ext-33zc}.

Repeated application of the operator $D^{-1}$ then gives
$$
D^{-(m+1)}g(x) ~=~ \int_1^x\int_1^{x_m}\cdots\int_1^{x_1}g(t)\,\, dt\, dx_1\,\cdots\, dx_m\, ,
$$
and the Cauchy formula for repeated integration asserts that
$$
D^{-(m+1)}g(x) ~=~ \int_1^x\frac{(x-t)^m}{m!}\, g(t)\, dt.
$$
This identity is obtained directly by applying Theorem~\ref{thm:71Tay32} (assuming $x>0$) to the function $f=D^{-(m+1)}g$ at $a=1$. For this reason, it is natural to refer to the result in Corollary~\ref{cor:72DiscTay34} as the \emph{discrete analog of Cauchy’s formula for repeated integration}.

\section*{Declaration of competing interest}

The authors declare that they have no conflict of interest.


\end{document}